\documentclass[11pt,a4paper]{article}

\usepackage{geometry}
\usepackage{enumitem}
\usepackage{tikz-cd}
\usepackage{amssymb}
\usepackage{amscd}
\usepackage{nccmath}
\usepackage{dsfont}
\numberwithin{equation}{section}
\usepackage{amsthm}
\usepackage{mathtools}
\usepackage[hidelinks]{hyperref}

\usepackage[style=alphabetic]{biblatex}
\renewbibmacro{in:}{}
\DeclareFieldFormat[article]{title}{#1\isdot}

\newcommand{\C}{\mathbb{C}}
\newcommand{\D}{\mathbb{D}^2}
\newcommand{\Z}{\mathbb{Z}}
\newcommand{\N}{\mathbb{N}}
\newcommand{\mane}{\text{Mañé }}
\newcommand{\di}{\text{d}}
\newcommand{\e}{\text{e}}
\renewcommand{\P}{\mathbb{P}}

\newcounter{exmp}
\newtheorem{preprop}[exmp]{Proposition}
\newtheorem{prelem}[exmp]{Lemma}
\newtheorem{theo}[exmp]{Theorem}
\theoremstyle{definition}
\newtheorem{rmk}[exmp]{Remark}

\title{Dimension of the Green measure of skew products\\ with dominant exponent in the fibers\thanks{This document is the English translation of \url{https://hal.science/hal-05779901v1} released on October 5, 2026}}
\author{Clara Genes}

\begin{document}
\maketitle
\begin{abstract}
\noindent Let $f(z,w)=(p(z),q(z,w))$ be a polynomial skew product of $\C^2$ of degree $d\geqslant 2$ and let $\mu$ and $\mu_p$ be the Green measures associated to $f$ and $p$ respectively. Let $\lambda_1 > \lambda_2$ and $\lambda_p$ be the Lyapunov exponents of $(f,\mu)$ and $(p,\mu_p)$ respectively. We show that if $\lambda_2 = \lambda_p$, then $\mu$ is exact dimensional and: 
\begin{equation*}
  \dim_H\mu
  =\frac{\log d}{\lambda_1}+\frac{\log d}{\lambda_2}.
\end{equation*}
The proof relies on the continuity of the conditional measures obtained by Jonsson, on the fact that $\mu$ has constant Jacobian $d^2$, and on a fibered version of a normal form theorem for the inverse branches of the iterates of $f$.
\end{abstract}
\section{Introduction}
\subsection{Main result} Let $f$ be a smooth map acting on a compact Riemannian manifold $M$ and let $\nu$ be an $f$-invariant measure on $M$. We define the lower and upper pointwise dimensions of $\nu$ by 
\begin{equation}\label{dimponct}
  \underline{\delta}_\nu(x) \coloneqq \liminf_{r\to0}\frac{\log\nu(B(x,r))}{\log r} \qquad \text{and} \qquad \overline{\delta}_\nu(x) \coloneqq \limsup_{r\to0}\frac{\log\nu(B(x,r))}{\log r}.
\end{equation} When these two quantities coincide, we call $\delta_\nu(x) \coloneqq \lim_{r\to 0} \log \nu(B(x,r))/\log r$ the pointwise dimension of $\nu$. By Young \cite{Young1982}, if $a\leqslant \delta_\nu\leqslant b$ holds $\nu$-almost everywhere, then the Hausdorff dimension of $\nu$ (defined as the infimum of the Hausdorff dimensions of sets of full $\nu$ measure) satisfies $a\leqslant \text{dim}_H(\nu)\leqslant b$. In particular, if the pointwise dimension exists and is constant $\nu$-almost everywhere, we have $\text{dim}_H(\nu) = \delta_\nu$. We refer to the book of Pesin \cite{Pesin1998} for an introduction to dimension theory in dynamical systems.

Given a dynamical system $(f,\nu)$, there are relations between the dimension of $\nu$, the Lyapunov exponents $\lambda_k \leqslant \cdots \leqslant \lambda_1$ of the system and its entropy $h_\nu$ (see \cite{Ledrappier1981, Pesin1998, LedrappierYoung1985}). In this paper, we are interested in holomorphic endomorphisms of $\mathbb{CP}^2$ of degree $d\geqslant 2$. These maps define non-invertible branched coverings of topological degree $d^2.$ We refer to the survey of Dinh-Sibony \cite{DinhSibony2010} for an overview of their dynamical properties. The question of the Hausdorff dimension of the equilibrium measure was raised by Fornæss-Sibony \cite{FornaessSibony2001}.

In the case of a rational map on $\mathbb{CP}^1$, Manning \cite{Manning1984}, in the hyperbolic case, then \mane \cite{Mane1988}, in the general case, proved the formula
\begin{equation} \label{manee}
    \delta = \frac{h_\nu}{\lambda} \quad \nu\text{-a.e.},
\end{equation}
for every ergodic measure satisfying $h_\nu > 0.$ Here, $\lambda = \int\log \lvert f'\lvert\di\nu$ denotes the unique Lyapunov exponent. The proof of \mane relies on Koebe's distortion Theorem and the Brin-Katok Theorem. In the case of an endomorphism of $\C\P^2$, a major difficulty is that the dynamical balls are no longer Riemannian balls because the ergodic invariant measures can have two different Lyapunov exponents.  Young \cite{Young1982} first proved, in the case of surfaces with $f$ a regular diffeomorphism and $\lambda_2 < 0 < \lambda_1$, the formula
\begin{equation*}\delta = \frac{h_\nu}{\lambda_1} - \frac{h_\nu}{\lambda_2}\quad \nu\text{-a.e.}.\end{equation*} Our main theorem is the following.
\begin{theo}\label{A}
     Let $f:(z,w)\in \C^2 \mapsto (p(z), q(z,w))\in \C^2$ be a polynomial skew product of algebraic degree $d\geqslant 2$. Let $\mu$ be the Green measure associated to $f$ and $\lambda_1 > \lambda_2$ its Lyapunov exponents such that $\lambda_2$ is the Lyapunov exponent of $(p, \mu_p)$. Then, for $\mu$-almost every $x\in\C^2$, we have 
     \begin{equation*}
        \lim_{r\to0}\frac{\log\mu\left(\D(x, r)\right)}{\log r} =  \frac{\log d}{\lambda_1}+\frac{\log d}{\lambda_2}.
    \end{equation*}
\end{theo}
Let us note that by Jonsson \cite{Jonsson1999} (see also Dupont-Taflin \cite{DupontTaflin2021}), the Lyapunov exponent of $(p, \mu_p)$, denoted $\lambda_p$ is one of the two exponents of the global system $(f,\mu)$. The proof of Theorem \ref{A} is outlined in Section \ref{intro1.2}. The lower bound $\delta_\mu \geqslant \frac{\log d}{\lambda_1}+\frac{\log d}{\lambda_2}$ $\mu$-almost everywhere was established by Dupont \cite{Dupont2011}. In this article we prove the other inequality. Let us note that the articles \cite{BinderDeMarco2003, DinhDupont2004, DupontRogue2020, BarreiraPesinSchmeling1999} and \cite{DeThelinVigny2015} obtain estimates of $\text{dim}_H(\mu)$ for general maps, which are not skew products. Our result shows in particular that the formula for $\text{dim}_H(\mu)$ conjectured by Binder-De Marco \cite{BinderDeMarco2003} is true in the case of skew products, when the expansion is dominant in the fibers. Being local, the same arguments of the proof of Theorem \ref{A} apply more generally to fibered endomorphisms of $\C\P^2$ which preserve a fibration \cite{DupontTaflin2021} along which the Lyapunov exponent of $\mu$ is maximal and equal to $\lambda_1$. In the general non-fibered case, it is known that if $\lambda_2 = \frac12\log d$, then $\dim_H\mu=\frac{\log d}{\lambda_1}+\frac{\log d}{\lambda_2},$ see \cite{DupontTaflin2021, DinhDupont2004}.
\subsection{Skew products and proof strategy}\label{intro1.2}
Skew products are maps $f:\C^2\to\C^2$ of the form 
\begin{equation*}
    f(z,w) = (p(z),q(z,w)),
\end{equation*}
where $p$ and $q$ are polynomials of the same degree $d\geqslant 2$ in $z$ and $w$ respectively such that $f$ extends to a holomorphic endomorphism of $\C\P^2$. An important property of skew products is that they preserve the vertical lines of $\C^2$: the restriction of $f^n$ to $\{z\}\times\C$ can be seen as the composition of $n$ polynomials of $\C$. Their dynamics along the fibers therefore generalizes in a non-autonomous way that of polynomials in dimension 1. In particular, we automatically have two dynamical systems: $(f,\mu)$ global system with $\mu$ the Green measure associated to $f$ and $(p,\mu_p)$ system in the base with $\mu_p = \pi_\star\mu$ the projection of the measure $\mu$ which is the Green measure associated to the polynomial $p$. The disintegration of $\mu$ (\cite{Rokhlin1949, Rokhlin1967}) gives a family of measures $\{\mu_z, \, z\in\C\}$ whose supports are in the vertical lines of $\C^2$:
\begin{equation}\label{desintegrintro}
    \mu=\int \mu_z\,\di\mu_p(z).
\end{equation}
In \cite{Jonsson1999}, Jonsson identified the measures $\mu_z$: they satisfy $\mu_z = \Delta_w G(z,w)$ with a proper normalization where $G$ is the Green function of $f$. This function being continuous on $\C^2$, it follows that $z\mapsto \mu_z$ is weak-* continuous.

The proof of Theorem \ref{A} relies on Jonsson's formula \eqref{desintegrintro} with the continuity of the family of conditional measures, the fact that the measure $\mu$ has constant Jacobian $d^2$ and on a normal form theorem of inverse branches, which we adapted to the skew product case. This theorem is central in our approach, its proof follows the strategy developed in \cite{BertelootDupontMolino2008, BertelootDupont2019, JonssonVarolin2002}.

Let us be more precise. For $x=(z,w)\in\C^2$ and $r>0$, we denote $\D(x,r)\coloneqq B(z,r)\times B(w,r)$. Since we have $B^2(x,r)  \subset \D(x,r) \subset B^2(x,\sqrt2r),$ the pointwise dimensions \eqref{dimponct} can be computed using bi-discs. Thus, in order to reach the dimension of $\mu$, we take a point $x_0=(z_0,w_0)\in\C^2$ generic for $\mu$, a small radius $r$ and we study the following equality: 
\begin{equation}
    \mu(\D(x_0,r)) = \int_{B(z_0,r)}\mu_z(B(w_0,r))\di\mu_p(z).\label{equa1}
\end{equation} 
If $z_0$ is generic for $\mu_p$, then according to the formula of \mane (\ref{manee}), $\mu_p(B(z_0,r))=r^{\log d/\lambda_p+o(1)}$. A technical lemma (Lemma \ref{cont0}) based on the weak-* continuity of $z\mapsto\mu_z$ gives a uniform lower bound 
\begin{equation}
    \mu_z(B(w_0,r))\geqslant c_0, \label{equa2}
\end{equation} 
for every $z$ close to $z_0$. According to the general normal form theorem: the inverse branch $f^{-m}_{\hat x_m}$ is well defined on the rectangle $R_m = B(z_m, r_0\e^{-m(\lambda_1-\lambda_2)})\times B(w_m, r_0)$. Thanks to our fibered version of this theorem, the inverse branch sends the rectangle $R_m$ to a bi-disc of the form $\D(x_0, \ell_m\e^{-m\lambda_1})$. We can now use the constant Jacobian of $\mu$ by injectivity of $f^m$ on this last bi-disc. Putting together these different arguments and \eqref{equa1}, we obtain \begin{equation*}
   \mu(\D(x_0,\ell_m\e^{-m\lambda_1})) \geqslant \mu(f^{-m}_{\hat x_m}(R_m)) = d^{-2m}\mu(R_m) \geqslant c_0d^{-2m}\left(\e^{-m(\lambda_1-\lambda_2)}\right)^{\log d/\lambda_2}.
\end{equation*}
This indeed gives $\mu\left(\D(x_0, \e^{-m\lambda_1})\right) \geqslant (\e^{-m\lambda_1})^{\frac{\log d}{\lambda_1}+\frac{\log d}{\lambda_2}}.$ This inequality extends to every radius $r$ by using upper and lower bounds. This proves Theorem \ref{A}. 
\paragraph{Acknowledgments.} I would like to thank my supervisors Christophe Dupont and Fabrizio Bianchi for their kindness and for their useful advice.
\paragraph{No AI declaration.} The author did not use any AI tools during the preparation of this article.
\section{Lyapunov exponents and fibered normal form theorem}\label{Section2}
Fix $f$ an endomorphism of $\C\P^2$. To overcome the lack of invertibility of the dynamical system $(f,\mu)$, we introduce its natural extension, see \cite[Chapter~10.4]{CornfeldFominSinai1982} for a more detailed construction. Let $O \coloneqq \left\{\hat{x}=(x_n)_{n\in\mathbb{Z}} \; ; \; x_{n+1}=f(x_n)\right\}$ be the set of orbits of $f$. We denote by 
\begin{equation*}
    \begin{array}{cccc}
        \tau : & O&\to &O \\
         & (\ldots,x_{-1},x_0,x_1,\ldots)&\mapsto &(\ldots,x_{-2},x_{-1},x_0,\ldots)
    \end{array}
\end{equation*} 
the right shift and by $\hat f = \tau^{-1}$ the left shift. Let $\hat{\mu}$ be the unique $\hat f$-invariant measure on $O$ satisfying $(\pi_0)_\ast \hat{\mu} = \mu,$ where $\pi_0(\hat{x}) \coloneqq x_0.$ Since $\mu$ is a probability measure and is mixing, the same holds for $\hat \mu$. We know that $\mu$ does not charge the critical set of $f$ \cite{DinhSibony2010}. In particular, the $\tau$-invariant subset $X \coloneqq \left\{\hat{x}\in O \; ; \; x_n \notin \text{Crit}(f),\ \forall n\in\Z\right\}$ satisfies $\hat{\mu}(X)=1.$ For every $\hat{x}\in X$, we denote by $f_{\hat{x}}^{-n}$ the inverse branch of $f^n$ which sends $x_0$ to $x_{-n}$. It is defined in a neighborhood of $x_0$ in $\C^2$ with values in a neighborhood of $x_{-n}$ in $\C^2$.
\subsection{Lyapunov exponents for skew products}
We denote 
\begin{equation*}
    \lambda_p\coloneqq\int\log|p'(z)|\di\mu_p(z)\quad \text{ and }\quad \lambda_q \coloneqq \int\log|q'_z(w)|\di\mu(z,w).
\end{equation*}
We note that $\lambda_p$ is the Lyapunov exponent of the system $(p,\mu_p)$. We have the following theorem.
\begin{theo}\label{Oseledets}
    Let $\lambda_1 > \lambda_2$ be the Lyapunov exponents of $(f,\mu)$ and $\lambda_p$ the Lyapunov exponent of $(p,\mu_p)$. Then $\{\lambda_1,\lambda_2\} = \{\lambda_p, \lambda_q\}.$ Moreover, if $\lambda_2 = \lambda_p$, then the Oseledets space associated to $\lambda_q$ is $\{0\}\times\C$. 
\end{theo}
\begin{proof}
    We know that $\log \lvert p'\lvert $ is integrable against $\mu_p$ since the latter admits a continuous Green potential. The Birkhoff theorem gives 
    \begin{equation*}
        \frac1n\sum_{k=0}^{n-1}\log|p'(p^k(z))| \overset{n\to\infty}{\longrightarrow}\lambda_p. 
    \end{equation*}
    for $\mu_p$-almost every $z\in\C$. Similarly, the sum $\Lambda_f \coloneqq \lambda_1+\lambda_2$ is defined for $\mu$-almost every $x\in\C^2$ by:
    \begin{equation*}
        \frac1n\sum_{k=0}^{n-1}\log|\det Df_{f^k(x)}| \overset{n\to\infty}{\longrightarrow}\Lambda_f \coloneqq \int\log|\det Df_x|\di\mu(x).
    \end{equation*}
    Since $f$ is a skew product, for every $(z,w)\in\C^2$, its Jacobian matrix is of the form 
    \begin{equation}
        Df_{(z,w)}=\begin{pmatrix}
            p'(z)&0\\\partial_zq(z,w)&\partial_wq(z,w)
        \end{pmatrix}.\label{jacobien}
    \end{equation}
    Thus, 
    \begin{equation*}
        \lambda_1+\lambda_2 = \Lambda_f = \int\log|\det Df_x|\di\mu(x) = \int\log|p'(z)|\di\mu_p(z) + \int\log|q'_z(w)|\di\mu(z,w) = \lambda_p+\lambda_q.
    \end{equation*}
    However, since the Lyapunov exponent of $\mu_p$ is one for $\mu$ (see \cite{Jonsson1999, DupontTaflin2021}), we obtain $\lambda_p\in\{\lambda_1,\lambda_2\}$. Thus $\{\lambda_1,\lambda_2\} = \{\lambda_p, \lambda_q\}.$ We set, for almost every $\hat x\in X$, 
    \begin{equation*}
        H_1(\hat x) = \text{vect}\begin{pmatrix}
        a(\hat x)\\b(\hat x)
        \end{pmatrix} \quad \text{ and }\quad  H_2 = \text{vect}\begin{pmatrix}
         0\\1
         \end{pmatrix}.
    \end{equation*} 
    with $a(\hat x)\neq0$. Let us observe that this gives a filtration $\{0\}\subsetneq H_1(\hat x)\subsetneq \text{vect}\left\{H_1(\hat x), H_2\right\} = \C^2.$ Moreover, for every $ v \in H_2\setminus\{0\}$, 
    \begin{equation*}
        \lim_{ n\to+\infty}\frac1{n}\log|Df^n(x_0)v|= \lambda_1 \quad \text{ and } \quad \lim_{ n\to+\infty}\frac1{n}\log|D(f^{-n}_{\hat x})(x_0)v|= -\lambda_1.
    \end{equation*}
    Indeed, from the form of the matrix (\ref{jacobien}) and since we have $\lambda_p = \lambda_2 < \lambda_1=\lambda_q$, we obtain by the Birkhoff theorem, for $\hat\mu$-almost every $\hat x \in X$,  
    \begin{equation*}   
        \frac1{n}\log\left|D(f^{-n}_{\hat x})(x_0)\begin{pmatrix} 0\\1\end{pmatrix}\right| = \frac1n\log \left\lvert\prod_{j=1}^{n}\left(\partial_wq(x_{-j})\right)^{-1}\right\lvert \overset{n\to\infty}{\longrightarrow}\int\log\left\lvert\partial_wq(x_{0})\right\lvert^{-1}\di\hat\mu(\hat x) = -\lambda_q = -\lambda_1.
    \end{equation*}
    The other limit is proved with the same Birkhoff argument. In particular, the Oseledets space associated to $\lambda_1$ is constant and given by $H_2$. Let us also note that, for $\hat \mu$-almost every $\hat x \in X$, there exist $a(\hat x)\neq 0$ and $b(\hat x)$ such that the Oseledets space associated to $\lambda_2$ is $H_1(\hat x)$. 
\end{proof}
\subsection{Normal form theorem and uniformization}
A function $\phi_\varepsilon : O \to ]0,1]$ (resp.  $O\to [1,+\infty[)$ is \emph{$\varepsilon$-slow} (resp. \emph{$\varepsilon$-fast}) if for every $\hat{x}\in O$ we have \begin{equation}\label{lente}
    \e^{-\varepsilon}\phi_\varepsilon(\hat{x})\leqslant\phi_\varepsilon(\tau(\hat{x}))\leqslant\e^{\varepsilon}\phi_\varepsilon(\hat{x}).
\end{equation} 
In an analogous way, we define the $\varepsilon$-slow sequences. We have the following normal form theorem.
\begin{theo}[Normal form  \cite{BertelootDupontMolino2008, BertelootDupont2019}]\label{formenormale}
    Let $f$ be a holomorphic endomorphism of $\C\P^2$ and let $\mu$ be its Green measure. Let $\lambda_2 <  \lambda_1$ be the Lyapunov exponents of $(f,\mu)$. Then for every $\varepsilon < \delta \ll \lambda_2$ and for $\hat{\mu}$-almost every $\hat{x} \in X$, there exist:
    \begin{itemize}
        \item an integer $n_{\hat{x}} \geqslant 1$ and real numbers $h_{\hat{x}} \geqslant 1$ and $0 < r_{\hat{x}}, \rho_{\hat{x}} \leqslant 1$,
        \item a sequence $(\varphi_{\hat{x}_n,k})_{k\geqslant0}$ of injective holomorphic maps:
            \begin{equation*}\begin{array}{ccccc}
                \varphi_{\hat{x}_n,k} &: &B\left(x_{n-k},\, r_{\hat{x}_n} \e^{-k(\delta+2\varepsilon)}\right) &\longrightarrow &\D\left(\rho_{\hat{x}_n} \e^{k\varepsilon}\right)   \end{array}
            \end{equation*}
            sending $x_{n-k}$ to $0$ and satisfying for every $u,v \in B\left(x_{n-k},\, r_{\hat{x}_n} \e^{-k(\delta+2\varepsilon)}\right),$
            \begin{equation}\label{distorglobale}
            e^{k(\delta-2\varepsilon)} \mathrm{dist}(u,v) \leqslant |\varphi_{\hat{x}_n,k}(u) - \varphi_{\hat{x}_n,k}(v)| \leqslant e^{k(\delta+3\varepsilon)} h_{\hat{x}_n} \mathrm{dist}(u,v),
            \end{equation}
        \item a sequence of linear maps  $(D_{\hat{x},k})_{k\geqslant0}$ represented by diagonal matrices of the form $\emph{diag}(a^2_{\hat x, k}, a^1_{\hat x, k})$ such that
            \begin{equation*}
                \e^{-k\lambda_i+k(\delta-\varepsilon)} \leqslant |a^i_{\hat{x},k}|\leqslant\e^{-k\lambda_i+k(\delta+\varepsilon)},
            \end{equation*}
    \end{itemize}
    for which the following diagram commutes for every $n\geqslant\hat{n}_{\hat{x}}$:
    \begin{equation*}
        \begin{CD}
            B(x_n, \,r_{\hat{x}_n}) @>{f_{\hat{x}_n}^{-n}}>>B\left(x_{0},\,r_{\hat{x}_n} \e^{-n(\delta+2\varepsilon)}\right)\\
            @V{\varphi_{\hat{x}_n,0}}VV @VV{\varphi_{\hat{x}_n,n}}V\\
            \D(\rho_{\hat{x}_n})@>>{D_{\hat{x}_n,n}}>\D\left(\rho_{\hat{x}_n} \e^{n\varepsilon}\right).
        \end{CD}
    \end{equation*}
    Moreover, the functions $\hat{x} \mapsto h_{\hat{x}}^{-1}, r_{\hat{x}}, \rho_{\hat{x}}$ are measurable and $\varepsilon$-slow on $X$ and the function $n_{\hat x}$ is measurable.
\end{theo}
In the setting of skew products, the maps $\varphi_{\hat x, n}$ are fibered. The following proposition describes this result. The proof is in Section \ref{Section3} and we mention that it remains valid for any skew product $f$ and any ergodic measure $\nu$ with positive Lyapunov exponents $\lambda_1>\lambda_2$.
\begin{preprop}\label{fibre}
    Let $\lambda_1 > \lambda_2$ be the Lyapunov exponents of $(f,\mu)$ such that $\lambda_2$ is the exponent of $(p,\mu_p)$. Then, the sequence $(\varphi_{\hat{x}_n,k})_{k\geqslant 0}$ given by Theorem \ref{formenormale} is fibered. More precisely, by setting $B_{\hat{x}_n,k} \coloneqq B\left(x_{n-k},\, r_{\hat{x}_n} \e^{-k(\delta+2\varepsilon)}\right)$, there exist measurable maps $\hat x\mapsto h'_{\hat x}, k_{\hat x}\geqslant 1$ and sequences of maps $(\psi_{1,k})_{k\geqslant 0}$, $(\psi_{2,k})_{k\geqslant 0}$ such that
    \begin{equation*}\begin{array}{ccccc}
        \varphi_{\hat{x}_n,k} &: &B_{\hat{x}_n,k} &\longrightarrow &\D\left(\rho_{\hat{x}_n} \e^{k\varepsilon}\right)\\
        &&(z,w)  &\longmapsto &(\psi_{1,k}(z), \psi_{2,k}(z,w))
    \end{array}\end{equation*}
    satisfies, for every $z,z' \in \pi\left(B_{\hat{x}_n,k}\right)$,
    \begin{equation}\label{distorbase}
        \e^{k(\delta-2\varepsilon)} \operatorname{dist}(z,z')\leqslant\left|\psi_{1,k}(z)-\psi_{1,k}(z')\right|\leqslant\e^{k(\delta+3\varepsilon)} h'_{\hat{x}_n}\, \operatorname{dist}(z,z'),
    \end{equation}
    and for every $(z_{n-k},u),(z_{n-k},u') \in B_{\hat{x}_n,k}$,
    \begin{equation}\label{distorfibre}
        \e^{k(\delta-2\varepsilon)} \operatorname{dist}(u,u')\leqslant\left|\psi_{2,k}(z_{n-k},u)-\psi_{2,k}(z_{n-k},u')\right|\leqslant\e^{k(\delta+3\varepsilon)} k_{\hat{x}_n}\operatorname{dist}(u,u').
    \end{equation}
\end{preprop}
\begin{rmk}
    The distortion in the fibers \eqref{distorfibre} will not be used: only the distortions in the base \eqref{distorbase} and the distortion \eqref{distorglobale} given by Theorem \ref{formenormale} will be used. Moreover, up to setting $h''_{\hat x} = \max\{h_{\hat x}\, , \, h'_{\hat x}\}$, we can assume $h_{\hat x} = h'_{\hat x}$.
\end{rmk}
From now on, we fix $0 < 2\varepsilon < \delta$ sufficiently small such that Theorem \ref{formenormale} is satisfied and such that $\lambda_1> 5\varepsilon+\delta$. Let $\hat{\Omega}_{L}$ be the set of full $\hat{\mu}$ measure given by Theorem \ref{formenormale}. Let $n_0, h_0, r_0$ be positive constants depending only on $\delta$ such that the set \begin{equation}\label{defomega0}
    \hat{\Omega}_0 = \left\{\hat{x}\in \hat{\Omega}_{L}\ ; \ n_{\hat x} \leqslant n_0, \ h_{\hat{x}}\leqslant h_0,\ r_{\hat{x}}\geqslant 2r_0\right\}
\end{equation} 
satisfies $\hat{\mu}(\hat{\Omega}_0)\geqslant 1-\delta$. According to \mane's theorem \cite{Mane1988}, for $\mu_p$-almost every  $z\not\in\bigcup_ip^{-i}(C(p))$, where $C(p)$ is the set of critical points of $p$, there exists $r_p(z)>0$ such that for every $r\leqslant r_p(z)$ we have $\mu_p(B(z,r)) \geqslant r^{\frac{\log d}{\lambda_p}+\varepsilon}.$ Let $S(\delta) = S >0$. We set
\begin{equation}
    \Omega_{S}^p\coloneqq \left\{z\not\in\bigcup_ip^{-i}(C(p))\,;\, r_p(z) \geqslant S\right\}.\label{equa6}
\end{equation}
We choose $S>0$ such that $$\hat{\Omega}_\delta\coloneqq \left\{\hat{x}\in \hat{\Omega}_0\ ; \ \pi\circ\pi_0(\hat{x})\in\Omega_S^p\right\}$$ has measure $\hat{\mu}(\hat{\Omega}_\delta)\geqslant 1-2\delta$.
\subsection{Positivity and continuity of the conditional measures}
First, let us show a general lemma which does not use the dynamics of the system.
\begin{prelem}\label{lempos}
    For $\mu$-almost every $x=(z,w)\in\C^2$ and for every $r>0$, we have $\mu_z(B^z(w,r)) >0.$
\end{prelem}
\begin{proof}
    By contradiction, suppose that there exists a measurable function $r$ of $x$, strictly positive on a measurable subset $\Omega'$ of strictly positive $\mu$ measure, such that for every $x\in \Omega'$, we have $\mu_z(B^z(w,r(x))) = 0.$ Let $\delta>0$ be sufficiently small such that
    \begin{equation*}
        \mu\left\{x\in\Omega'\, ;\, r(x)> 2\delta\right\} \geqslant \frac{\mu(\Omega')}{2}>0.
    \end{equation*}By the Borel density lemma, there exists a mesurable set $D\subset \left\{x\in\Omega'\,;\, r(x)>2\delta\right\}$ such that 
    \begin{equation*}
        \mu(D) = \mu\left(\left\{r(x)>2\delta\right\}\right)\quad \text{ and } \quad \forall x\in D,\quad \frac{\mu(\D(x,\rho)\cap D)}{\mu(\D(x,\rho))}\overset{\rho\to 0}{\longrightarrow}1.
    \end{equation*}
    Thus, if $x=(z,w)\in D$, there exists $\rho(x)<\delta$ such that 
    \begin{equation*}
        \mu(\D(x,\rho(x))) \leqslant 2\mu(\D(x,\rho(x))\cap D) = 2\int_{z'\in B(z,\rho(x))}\mu_{z'}(B^{z'}(w,\rho(x))\cap D)\di\mu_p(z').
    \end{equation*}
    Let $z'\in B(z,\rho(x))$. Either $B^{z'}(w,\rho(x))\cap D = \emptyset,$ and then $\mu_{z'}(B^{z'}(w,\rho(x))\cap D) = 0.$ Or, there exists $x'=(z',w')\in B^{z'}(w,\rho(x))\cap D$. But then in particular $x'\in D$ and we have the following sequence of inequalities: $\rho(x) < \delta < 2\delta < r(x')$. Therefore, we have the following series of inclusions $B^{z'}(w,\rho(x))\subset B^{z'}(w',2\delta) \subset B^{z'}(w',r(x'))$ where the first inclusion follows from the triangle inequality. Therefore 
    \begin{equation*}
        \mu_{z'}(B^{z'}(w,\rho(x))\cap D) \leqslant \mu_{z'}(B^{z'}(w',r(x')))=0,
    \end{equation*}
    where the last equality again follows from the fact that $x'\in D.$ Thus, we have $\mu(\D(x,\rho(x)))=0$ which is absurd since $x$ is generic for $\mu$.
\end{proof}
We denote by $\Omega_{>0}$ the set of full $\mu$ measure given by Lemma \ref{lempos}. We denote $\hat{\Omega}_{>0}\coloneqq \pi_0^{-1}(\Omega_{>0})$. For $r>0$ and $\hat x\in \hat{\Omega}_{>0}$ with $\pi_0(\hat x) = x_0 = (z_0,w_0)$, we know by Lemma \ref{lempos} that there exists $c_0(x_0)>0$ such that $\mu_{z_0}(B^{z_0}(w_0,r))\geqslant c_0(x_0)$. There therefore exists $c_0(\delta, r)>0$ sufficiently small such that $$\hat{\Omega}_{c_0} \coloneqq \{\hat x\in \hat{\Omega}_{>0}\, ;\, \mu_{z_0}(B^{z_0}(w_0,r))> c_0(\delta,r)\}$$ has measure $\hat\mu(\hat{\Omega}_{c_0}) \geqslant 1-\delta$. We prove the following lemma, which will allow us to extend local properties true on a fiber $z_0$ to neighboring fibers $z$ for small perturbations.
\begin{prelem}\label{cont0}
    Let $\delta>0$. There exists a real $\gamma_0 = \gamma_0(\delta,r_0)>0$ and a set $\hat{\Omega}_{cont}(\delta,r_0)\subset X$ of measure $\hat{\mu}(\hat{\Omega}_{cont})\geqslant 1-2\delta$ such that for every $\hat{x}\in \hat{\Omega}_{cont}$, denoting $\pi_0(\hat{x}) = x_0=(z_0,w_0)$, we have for every $z\in B(z_0,\gamma_0)$,  \begin{equation}\label{eqcont0} 
        \mu_z(B^z(w_0,r_0)) \geqslant  \frac12\mu_{z_0}(B^{z_0}(w_0,r_0)).
    \end{equation}
\end{prelem}
\begin{proof}
    Fix $\delta>0$ and $\hat x\in \hat{\Omega}_{c_0}$. By definition of $\hat\Omega_{c_0}$, we know that there exists $c_0(\delta, r_0)>0$ such that $\mu_{z_0}(B^{z_0}(w_0, r_0)) >c_0(\delta, r_0)$ by disintegration. Since the ball $B^{z_0}(w_0, r_0)$ is open, the Portmanteau Theorem gives \begin{equation*}
        \liminf_{z\to z_0}\mu_z(B^{z}(w_0, r_0)) \geqslant \mu_{z_0}(B^{z_0}(w_0,r_0)).
    \end{equation*} since the family of measures $\mu_z$ is continuous. Therefore, there exists a positive measurable function $\gamma_r(\hat x)$ such that, for every $z\in B(z_0, \gamma_{r_0}(\hat x))$, we have $\mu_z(B^z(w_0, r_0)) >\frac12 \mu_{z_0}(B^{z_0}(w_0,r_0))$. We fix a constant $0 < \gamma_0 < r_0$ chosen such that 
    \begin{equation*}
        \hat{\Omega}_{cont}(\delta,r_0) \coloneqq\left\{\hat{x}\in\hat{\Omega}_{c_0} \,;\, \gamma_{r_0}(\hat x)\geqslant \gamma_0\right\}
    \end{equation*}
    has measure $\hat{\mu}(\hat{\Omega}_{cont}) \geqslant 1-2\delta.$
\end{proof}
We set $\hat{\Omega}_1\coloneqq \hat{\Omega}_{cont}(\delta,r_0)\cap\hat{\Omega}_\delta$. This set has measure $\hat{\mu}(\hat{\Omega}_1)\geqslant 1-5\delta$ and satisfies in particular $\hat\Omega_1\subset \hat\Omega_{c_0(\delta, r_0)}$.
\subsubsection{Return to the good set}
We recall that $\hat{f} =\tau^{-1}$ denotes the left shift in $X$ and that $\hat{\mu}$ is $\hat f$-invariant. We denote $\hat x_m\coloneqq \hat f^m(\hat x).$
\begin{prelem}\label{exim}
    There exists $\hat{\Omega}\subset \hat{\Omega}_1 $ such that $\hat{\mu}(\hat{\Omega})=\hat{\mu}(\hat{\Omega}_1)$ and for every $\hat{x}\in \hat{\Omega}$, there exists $n_1(\hat{x})>0$ such that for every $n\geqslant n_1(\hat{x})$, there exists $m(n,\hat{x})=m$ such that $n(1-7\delta)\leqslant m\leqslant n$ and $\hat{f}^m(\hat{x})\in\hat{\Omega}_1.$
\end{prelem}
\begin{proof}
    According to the Birkhoff theorem, there exists $\hat{\Omega}\subset \hat{\Omega}_1$ such that $\hat{\mu}(\hat{\Omega}) = \hat{\mu}(\hat{\Omega}_1)$ and such that for every $\hat{x}\in\hat{\Omega}$, 
    \begin{equation*}
        \frac1n\sum_{k=0}^{n-1}\mathds{1}_{\hat{\Omega}_1}(\hat f^k(\hat x)) \overset{n\to\infty}{\longrightarrow}\hat{\mu}(\hat{\Omega}_1).
    \end{equation*}
    By the pigeonhole principle, there exists $n_1(\hat x)$ such that for every $n\geqslant n_1(\hat x)$, there exists $m(n,\hat{x})=m$ such that $n(1-7\delta)\leqslant m\leqslant n$ and $\hat{f}^m(\hat{x})\in\hat{\Omega}_1$.
\end{proof}
\subsubsection{Definition of a good bidisc}
\begin{prelem}\label{ellm}
    Let $\hat{x}\in \hat{\Omega}$ with $x_0 = \pi_0(\hat x)$, $n\gg1$ and $m(n, \hat x) = m$ be given by Lemma \ref{exim}. We denote $f^m(x_0)= x_m = (z_m, w_m)$ where $z_m = p^m(z_0)$ and $w_m = Q^m_{z_0}(w_0)$. We have $$f^{-m}_{\hat x_m}(R_m) \subset \D(x_0,\ell_m\e^{-m\lambda_1}),$$ where $R_m \coloneqq B(z_m, r_0\e^{-m(\lambda_1-\lambda_2)})\times B(w_m,r_0)$ and $\ell_m \coloneqq \e^{3m\varepsilon}r_0h_{0}$.
\end{prelem}
\begin{proof}
    By the definition \eqref{defomega0} of $r_0$, we indeed have $R_m \subset B(x_m, r_{\hat x_m}).$ Thus, according to the normal form theorem \ref{formenormale}, the inverse branch $f^{-m}_{\hat x_m}$ and $\varphi_{\hat x_m, 0}$ are well defined on $R_m$. By the fibered properties of $\varphi_{\hat x_m, 0}$ we have 
    \begin{equation*}
        \varphi_{\hat x_m, 0}(R_m) \subset B(0,r_0\e^{-m(\lambda_1-\lambda_2)}h_{\hat x_m})\times B(0, r_0h_{\hat x_m}).
    \end{equation*}
    We can then apply the linear map $D_{\hat x_m, m}$ and we obtain 
    \begin{align*}
        D_{\hat x_m, m}\circ \varphi_{\hat x_m, 0}(R_m) &\subset B(0,a^2_{\hat x_m, m}r_0\e^{-m(\lambda_1-\lambda_2)}h_{\hat x_m})\times B(0, a^1_{\hat x_m, m}r_0h_{\hat x_m})\\ &\subset B(0,\e^{-m(\lambda_2-\delta-\varepsilon)}r_0\e^{-m(\lambda_1-\lambda_2)}h_{\hat x_m})\times B(0, \e^{-m(\lambda_1-\delta-\varepsilon)}r_0h_{\hat x_m})\\&\subset \D(0, \e^{-m(\lambda_1-\delta-\varepsilon)}r_0h_{0}),
    \end{align*}
    where $h_{\hat x_m}\leqslant h_0$ since $\hat x_m\in\hat\Omega_0$. From the distortion estimate \eqref{distorglobale} given by the normal form theorem applied to a point on the boundary $u\in\partial B\left(x_0, r_{\hat x_m}\e^{-m(\delta+2\varepsilon)}\right)$, we have $$\lvert \varphi_{\hat x_m, m}(u)\lvert \geqslant \e^{m(\delta-2\varepsilon)}r_{\hat x_m}\e^{-m(\delta+2\varepsilon)} = r_{\hat x_m}\e^{-4m\varepsilon}.$$ This implies, since $\varphi_{\hat x_m, m}$ sends $x_0$ to 0 and is injective, that the image $\varphi_{\hat x_m, m}\left(B\left(x_{0},r_{\hat{x}_m} \e^{-m(\delta+2\varepsilon)}\right)\right)$ is contained in a ball centered at 0 and of radius comparable to $r_{\hat x_m}\e^{-4m\varepsilon}$. Recall that we have chosen $\varepsilon,\delta$ sufficiently small such that $\lambda_1-\delta-\varepsilon > 4\varepsilon$. Then, for $n$ sufficiently large we have $ \e^{-m(\lambda_1-\delta-\varepsilon)}r_0h_0\ll r_0\e^{-4m\varepsilon}$. This gives the inclusion 
    \begin{equation}
        \D(0, \e^{-m(\lambda_1-\delta-\varepsilon)}r_0h_0) \subset \varphi_{\hat x_m, m}\left(B\left(x_{0},\,r_{\hat{x}_m} \e^{-m(\delta+2\varepsilon)}\right)\right) \subset \D(\rho_{\hat x_m}\e^{m\varepsilon}).
    \end{equation}
    In particular, by injectivity of $\varphi_{\hat x_m, m}$, we can consider its local inverse and we obtain, again thanks to the distortion \eqref{distorglobale}: 
    \begin{align*}
        f^{-m}_{\hat x_m}(R_m) = \varphi_{\hat x_m, m}^{-1}\circ D_{\hat x_m, m}\circ \varphi_{\hat x_m, 0}(R_m) \subset \D(x_0, \e^{-m(\delta-2\varepsilon)}\e^{-m(\lambda_1-\delta-\varepsilon)}r_0h_{0}).
    \end{align*}
    Thus, we have $f^{-m}_{\hat x_m}(R_m) \subset \D(x_0, \e^{-m(\lambda_1-3\varepsilon)}r_0h_{0})$.
\end{proof}
\subsection{Proof of Theorem \ref{A}}
We can now prove Theorem \ref{A}. Let $\hat{x}\in \hat{\Omega}$, $n\gg1$ and $m(\hat x, n)$ be defined by Lemma \ref{exim}. We choose $n$ large enough such that $r_0\e^{-n(1-7\delta)(\lambda_1-\lambda_2)}\leqslant \min\{\gamma_0\,,\,S\}$. By Lemma \ref{ellm}, we have $$\D(x_0, \ell_m\e^{-m\lambda_1})\supset f^{-m}_{\hat x_m}(R_m).$$ Thus, by iterating $m$ times the relation $f^\star\mu = d^2\mu$ and by injectivity of $f^m$ on $f^{-m}_{\hat x_m}(R_m)$, we have 
\begin{equation*}
    \mu(\D(x_0, \ell_m\e^{-m\lambda_1}))\geqslant  d^{-2m}\mu(R_m).
\end{equation*}
But then, by disintegration
\begin{equation*}
    \mu(\D(x_0, \e^{-m\lambda_1}\ell_m)) \geqslant d^{-2m} \int_{z\in B(z_m, r_0\e^{-m(\lambda_1-\lambda_2)})}\mu_z(B^z(w_m, r_0))\di\mu_p(z).
\end{equation*}
Since $\hat x\in \hat{\Omega}$, in particular $\hat f^m(\hat x)\in \hat{\Omega}_1 \subset \hat{\Omega}_{cont}$ and we can apply Lemma \ref{cont0} to it, which gives 
\begin{equation*}
    \mu(\D(x_0, \e^{-m\lambda_1}\ell_m)) \geqslant d^{-2m} \frac12\mu_{z_m}(B^{z_m}(w_m, r_0))\mu_p(B(z_m, r_0\e^{-m(\lambda_1-\lambda_2)})).
\end{equation*} 
Moreover, since $\hat f^m(\hat x)\in \hat{\Omega}_1 \subset \hat{\Omega}_{\delta}$, we have $r_p(z_m)\geqslant S \geqslant r_0\e^{-m(\lambda_1-\lambda_2)}$, where $r_p$ is defined in (\ref{equa6}). Therefore 
\begin{equation*}
    \mu(\D(x_0, \e^{-m\lambda_1}\ell_m)) \geqslant \frac12d^{-2m} \mu_{z_m}(B^{z_m}(w_m, r_0))(r_0\e^{-m(\lambda_1-\lambda_2)})^{\log d /\lambda_p+\varepsilon}.
\end{equation*}
Finally, since $\hat f^m(\hat x)\in\hat\Omega_1 \subset \hat\Omega_{c_0}$ and $\lambda_p= \lambda_2$, setting $c_0 \coloneqq \frac12c_0(\delta, r_0)$
\begin{equation*}
     \mu(\D(x_0, \e^{-m\lambda_1}\ell_m)) \geqslant d^{-2m}c_0(r_0\e^{-m(\lambda_1-\lambda_2)})^{\log d/\lambda_2+\varepsilon}.
\end{equation*}
After simplifications, this gives
\begin{equation*}
    \log\mu(\D(x_0, \e^{-m\lambda_1}\ell_m)) \geqslant -m\left(\varepsilon(\lambda_1-\lambda_2)+ \left(\frac{\lambda_1}{\lambda_2}+1\right)\log d\right)+\log \left(c_0r_0^{\log d/\lambda_2+\varepsilon}\right).
\end{equation*}
Since $n(1-7\delta)\leqslant m(n,\hat x)\leqslant n$ and $\e^{-m\lambda_1}\ell_m = \e^{-m(\lambda_1-3\varepsilon)}r_0h_0$, we have
\begin{equation*}
    \frac{\log\mu(\D(x_0, \e^{-n(1-7\delta)(\lambda_1-3\varepsilon)}r_0h_0))}{\log \e^{-n(\lambda_1-3\varepsilon)}r_0h_0} \leqslant \frac{\log\mu(\D(x_0, \e^{-m\lambda_1}\ell_m))}{\log \e^{-m\lambda_1}\ell_m}.
\end{equation*}
Hence,
\begin{equation*}
    \frac{\log\mu(\D(x_0, \e^{-m\lambda_1}\ell_m))}{\log \e^{-m\lambda_1}\ell_m} \leqslant  \frac{1}{\lambda_1-3\varepsilon-\frac{\log r_0h_0}{m}}\left(\varepsilon(\lambda_1-\lambda_2)+ \left(\frac{\lambda_1}{\lambda_2}+1\right)\log d\right)-\frac{\log \left(c_0r_0^{\log d/\lambda_2+\varepsilon}\right)}{m(\lambda_1-3\varepsilon)-\log r_0h_0}.
\end{equation*}
Set $\rho_n \coloneqq \e^{-n(1-7\delta)(\lambda_1-3\varepsilon)}r_0h_0$. Passing to the upper limit, we then obtain:
\begin{align*}
    \limsup_{n\to\infty}\frac{\log\mu(\D(x_0, \rho_n))}{\log \rho_n} &\leqslant  \limsup_{n\to\infty}\frac{ \log \e^{-n(\lambda_1-3\varepsilon)}r_0h_0}{\log \rho_n}\frac{\log\mu(\D(x_0, \e^{-m\lambda_1}\ell_m))}{\log \e^{-m\lambda_1}\ell_m} \\&\leqslant \frac{1}{(1-7\delta)(\lambda_1-3\varepsilon)}\left(\varepsilon(\lambda_1-\lambda_2)+ \left(\frac{\lambda_1}{\lambda_2}+1\right)\log d\right).
\end{align*}
Since the sequence $(\rho_n)_{n\in\N}$ is decreasing and satisfies $\lim_{n\to\infty}\rho_n = 0$, for every $r>0$ small, there exists $n\geqslant 0$ such that $\rho_{n+1}\leqslant r \leqslant \rho_n$. Therefore, the inequality $\mu(\D(x_0,\rho_{n+1}))\leqslant \mu(\D(x_0,r))$ and the negativity of the logarithms give: 
\begin{equation*}
    \frac{\log\mu(\D(x_0, r))}{\log r}\leqslant \frac{\log\mu(\D(x_0, \rho_{n+1}))}{\log \rho_{n+1}}\frac{\log\rho_{n+1}}{\log r}.
\end{equation*}
Since, 
\begin{equation*}
   1 \leqslant\frac{\log\rho_{n+1}}{\log r} \leqslant \frac{\log\rho_{n+1}}{\log\rho_n} \overset{n\to\infty}{\longrightarrow} 1,
\end{equation*}
we finally obtain that for every $\hat{x}\in\hat{\Omega}$, 
\begin{equation}
    \limsup_{r\to 0}\frac{\log\mu\left(\D(x_0,r)\right)}{\log r} \leqslant \frac{1}{(1-7\delta)(\lambda_1-3\varepsilon)}\left(\varepsilon(\lambda_1-\lambda_2)+ \left(\frac{\lambda_1}{\lambda_2}+1\right)\log d\right).\label{eqpr}
\end{equation}
We denote $\hat \Omega = \hat \Omega(\delta)$ to recall that this set depends only on $\delta$. Set $$\Lambda\coloneqq \bigcap_{j\geqslant1}\bigcup_{k\geqslant j}\hat{\Omega}\left(\frac1k\right).$$ We have $\hat{\mu}(\Lambda) = 1$ since $\hat{\mu}(\hat{\Omega}\left(\frac1k)\right)\geqslant 1-7/k$ for every $k\geqslant 1$. For every $\hat{x}\in \Lambda$, there exists a sequence $(k_i(\hat{x}))_{i\geqslant 1}$ such that $k_i(\hat x) \to+\infty$ and $\hat{x}\in \hat{\Omega}(\frac1{k_i(\hat{x})})$ for every $i\geqslant 1$. Replacing $\delta$ by $1/k_i(\hat{x})$ in (\ref{eqpr}) and choosing for every $\delta$ an $\varepsilon(\delta)<\delta/2$, we obtain, by taking the limit $i\to\infty$
\begin{equation*}
    \limsup_{r\to 0}\frac{\log\mu\left(\D(x_0,r)\right)}{\log r} \leqslant \frac{1}{\lambda_1}\left(\frac{\lambda_1}{\lambda_2}+1\right)\log d = \frac{\log d}{\lambda_1} + \frac{\log d}{\lambda_2}.
\end{equation*}
Combining this inequality with the lower limit given by Dupont \cite{Dupont2011}, we obtain that the Green measure $\mu$ is exact dimensional equal to $\frac{\log d}{\lambda_1} + \frac{\log d}{\lambda_2}$, which finishes the proof of Theorem \ref{A}.
\section{Proof of the theorem of fibered normal forms}\label{Section3}
This section proves Proposition \ref{fibre}. Recall that $\tau$ is the right shift on $X$ and that $\hat f = \tau^{-1}$ is the left one, see Section \ref{Section2} for other notations on the natural extension.
\subsection{Oseledets-Pesin Theorem}\label{section3.1}
For a function $a : X \to ]0,+\infty[$, we set $E(a) \coloneqq \bigcup_{\hat{x}\in X}\{\hat{x}\}\times \D(a(\hat x))$. Since $f$ is holomorphic on $\C^2$, there exist constants $0<M_1 \leqslant M_0$ such that the map
\begin{equation*}
   \begin{array}{ccccc} \mathcal{F} &:  &E(M_1) & \longrightarrow & E(M_0) \\
    &&(\hat{x}, v) & \longmapsto & \left(\hat{f}(\hat{x}),\, f(x_0+v)-f(x_0)\right)
    \end{array}
\end{equation*}
is well defined, where $x_0 = \pi_0(\hat x)$. Note that, by definition of $X$, for every $\hat x\in X$ the map $f$ is locally invertible near $x_0$, and we denote by $f^{-1}_{\hat x}$ its inverse branch sending $x_0$ to $x_{-1} \coloneqq \pi_0(\tau(\hat x))$. In \cite{BertelootDupont2019} and \cite{BertelootDupontMolino2008}, it is shown that for every $\varepsilon > 0$ there exists an $\varepsilon$-slow function $\rho_\varepsilon : X \to (0, 1]$ such that the following map is well defined:
\begin{equation*}
    \begin{array}{ccccc} \mathcal{F}^{-1} &: &E(\rho_\varepsilon) & \longrightarrow & E(M_1) \\
    &&(\hat{x}, v) & \longmapsto & \left(\tau(\hat{x}),\, g_{\hat x}(v)\right)
    \end{array} \quad \text{ where } \quad  g_{\hat x}(v) \coloneqq f^{-1}_{\hat{x}}(x_0+v)-x_{-1}.
\end{equation*}
For $i=1,2$, we set $\hat{L}_i \coloneqq \bigcup_{\hat{x} \in X} \{\hat{x}\} \times L_i$, where $L_1 = \{0\} \times\C$ and $L_2 = \C\times \{0\}$, so that $\cup_{\hat{x} \in X} \{\hat{x}\} \times \C^2 = \hat{L}_1 \oplus \hat{L}_2$. The Oseledets-Pesin theorem can be stated as follows (see \cite[Theorem~S.2.10]{KatokHasselblatt1995}).
\begin{theo}[Oseledets-Pesin, \cite{KatokHasselblatt1995}]\label{2.2}
    Let $d_0\mathcal{F}^{-1}$ be the linear part of $\mathcal{F}^{-1}$, defined by $d_0\mathcal{F}^{-1}(\hat x, v) = (\tau(\hat x), d_0g_{\hat x}(v))$. For every $\varepsilon > 0$, there exist an invertible linear map $\mathcal{C}_\varepsilon$ such that $\mathcal{C}_\varepsilon(\hat x, v) = (\hat x, \mathcal{C}_{\varepsilon,\hat x}(v))$ and an $\varepsilon$-fast function $h_\varepsilon : X \to [1, +\infty[$ such that
    \begin{enumerate}
        \item the linear map $\mathcal{A} \coloneqq \mathcal{C}_\varepsilon \circ d_0\mathcal{F}^{-1} \circ \mathcal{C}_\varepsilon^{-1}$ stabilizes $\hat{L}_1$ and $\hat{L}_2$, and is of the form $\mathcal{A}(\hat x, v) = (\tau(\hat x), A_{\hat{x}}v)$ with $A_{\hat{x}} = \operatorname{diag}(a_2(\hat{x}), a_1(\hat{x}))$, where each $a_i(\hat{x})$ satisfies $\e^{-\lambda_i - \varepsilon} \leqslant |a_i(\hat{x})| \leqslant \e^{-\lambda_i + \varepsilon};$
        \item for $\hat\mu$-almost every $\hat{x} \in X$ and every $v \in \C^2$, $|v| \leqslant|\mathcal{C}_{\varepsilon, \hat{x}}(v)| \leqslant h_\varepsilon(\hat{x})\,|v|$.
    \end{enumerate}
\end{theo}
\begin{preprop}\label{Cepsi}
    Let $f$ be a skew product such that $\lambda_2 = \lambda_p$. Then, for $\hat \mu$-almost every $\hat x\in X$, the linear map $\mathcal{C}_{\varepsilon,\hat x}$ given by Theorem \ref{2.2} can be chosen fibered.
\end{preprop}
\begin{proof}
    We follow the construction of $\mathcal{C}_\varepsilon$ in the proof of \cite[Theorem~S.2.10]{KatokHasselblatt1995}. By Theorem \ref{Oseledets}, for every $\hat x \in X$ there exist $a(\hat x)\neq 0$ and $b(\hat x)$ such that the Oseledets spaces at $\hat x$ are given by the span of
    \begin{equation*}
        H_1(\hat x) = \begin{pmatrix} a(\hat x)\\ b(\hat x) \end{pmatrix} \quad\text{and}\quad H_2(\hat x) = \begin{pmatrix} 0\\ 1 \end{pmatrix},
    \end{equation*}
    where $H_1$ is associated to $\lambda_2$ and $H_2$ to $\lambda_1$. We normalize $\lVert H_1(\hat x)\rVert = 1$ (the equality $\lVert H_2(\hat x)\rVert = 1$ is obvious). Since these lines are invariant, there exists $\xi_i(\hat x)\in\C$ with $|\xi_i(\hat x)|=1$ such that 
    \begin{equation}\label{d0g}
        \frac{d_0 g_{\hat{x}}(H_i(\hat{x}))}{\left\| d_0 g_{\hat{x}}(H_i(\hat{x})) \right\|} = \xi_i(\hat{x})\, H_i(\tau(\hat{x})).
    \end{equation}
    Writing $d_0 g^m_{\hat x}\coloneqq d_0 g_{\tau^{m-1}(\hat x)}\circ\cdots\circ d_0g_{\hat x}$, (and the inverse of such compositions when $m$ is negative), we set
    \begin{equation}\label{deftheta}
        \theta_1(\hat{x}) = \sum_{m \in \Z} \left\| d_0 g^{m}_{\hat{x}}(H_1(\hat{x})) \right\|^2 \e^{2m\lambda_2}\, \e^{-2\varepsilon |m|}\quad \text{ and } \qquad \theta_2(\hat{x}) = \sum_{m \in \Z} \left\| d_0 g^{m}_{\hat{x}}(H_2(\hat{x})) \right\|^2 \e^{2m\lambda_1}\, \e^{-2\varepsilon |m|}.
    \end{equation}
    These series converge (see \cite{KatokHasselblatt1995}) and satisfy $\theta_i \geqslant 1$ (sum of positive terms with the term $m=0$ that is equal to $1$). Re-indexing the sum defining $\theta_1(\tau(\hat x))$ and using $\e^{-2\varepsilon}\leqslant \e^{-2\varepsilon(|k-1|-|k|)}\leqslant\e^{2\varepsilon}$, we get
    \begin{equation}\label{theta}
        \theta_1(\tau(\hat{x}))^{1/2}= \theta_1(\hat{x})^{1/2} \left\| d_0 g_{\hat{x}}(H_1(\hat{x})) \right\|^{-1}\theta\e^{-\lambda_2},\qquad \e^{-\varepsilon} \leqslant \theta \leqslant \e^{\varepsilon}.
        \end{equation}
    and similarly for $\theta_2$. Define
    \begin{equation}\label{defC}
        \begin{array}{ccccl}
            \mathcal{C}_{\varepsilon,\hat{x}}^{-1} & :&\C^2&\longrightarrow &\C^2 \\
             &&(z, w) &\longmapsto &\frac{1}{\sqrt{2}}\left(z\, \theta_1(\hat{x})^{-1/2} H_1(\hat{x}) + w\, \theta_2(\hat{x})^{-1/2} H_2(\hat{x})\right).
        \end{array}
    \end{equation}
    Expanding, we obtain the following fibered map
    \begin{equation}\label{defCextented}
        \mathcal{C}_{\varepsilon,\hat{x}}^{-1}(z, w) = \frac{1}{\sqrt{2}}\left( z\, \theta_1(\hat{x})^{-1/2} a(\hat{x}),\; z\, \theta_1(\hat{x})^{-1/2} b(\hat{x}) + w\, \theta_2(\hat{x})^{-1/2} \right).
    \end{equation}
    Since $a(\hat x)\neq0$, the map $\mathcal{C}_{\varepsilon,\hat{x}}$ is well defined as the inverse of $\mathcal{C}_{\varepsilon,\hat{x}}^{-1}$ and is also fibered. Finally, combining \eqref{d0g} and \eqref{theta}, we get
    \begin{align*}
        d_0 g_{\hat{x}}\big(\mathcal{C}_{\varepsilon,\hat{x}}^{-1}(1,0)\big) &= \frac{1}{\sqrt 2}\theta_1^{-1/2}(\hat x)d_0g_{\hat x}(H_1(\hat x)) \\&= \frac{1}{\sqrt 2}\xi_1(\hat x)\|d_0g_{\hat x}(H_1(\hat x))\|\theta^{-1/2}_1(\hat x)H_1(\tau(\hat x))\\&= \xi_1(\hat x)\theta\e^{-\lambda_2}\mathcal{C}_{\varepsilon,\tau(\hat{x})}^{-1}(1,0),
    \end{align*}
    and similarly for $(0,1)$. Thus
    \begin{equation*}
        \mathcal{C}_{\varepsilon,\tau(\hat{x})} \circ d_0 g_{\hat{x}} \circ \mathcal{C}_{\varepsilon,\hat{x}}^{-1} = \operatorname{diag}\big(a_2(\hat{x}), a_1(\hat{x})\big),
    \end{equation*}
    where $\e^{-\lambda_i - \varepsilon} \leqslant |a_i(\hat{x})| \leqslant \e^{-\lambda_i + \varepsilon}$. Moreover, we verify easily that $\lvert\mathcal{C}^{-1}_{\varepsilon,\hat x}(v)\lvert \leqslant \lvert v\vert$ form \eqref{defC}. Resolving a system from \eqref{defCextented} and using that $\lvert a(\hat x)\lvert ,\, \lvert b(\hat x)\lvert \leqslant 1$, we get 
    \begin{equation*}
        \lvert \mathcal{C}_{\varepsilon, \hat x}(v)\lvert \leqslant \lvert v\lvert\, \frac{1}{\lvert a(\hat x)\lvert }\left(\theta_1^{1/2}+2\theta_2^{1/2}\right).
    \end{equation*}
    The functions $\theta_1$, $\theta_2$ are tempered by their definition \eqref{deftheta} and using the Birkhoff theorem. The function $|a|^{-1}$ is also tempered since the angle between the Oseledets spaces is always tempered (see \cite{KatokHasselblatt1995,BarreiraPesinSchmeling1999}). Hence, $h:\hat x\mapsto \frac{1}{\lvert a(\hat x)\lvert }\left(\theta_1^{1/2}+2\theta_2^{1/2}\right)$ is tempered. Thus, there exist an $\varepsilon$-fast function $h_\varepsilon$ such that $h \leqslant h_\varepsilon$ for $\hat\mu$-almost every point $\hat x\in X$ (see \cite{BertelootDupont2019}) and this gives the upper bound of $\mathcal{C}_{\varepsilon, \hat x}$.
\end{proof}
Since $\mathcal{F}^{-1} : E(\rho_\varepsilon) \to E(M_1)$, by the point 2 of Theorem \ref{2.2}, the following map is well defined:
\begin{equation}\label{defcalW}
    \mathcal{W} \coloneqq \mathcal{C}_\varepsilon \circ \mathcal{F}^{-1} \circ \mathcal{C}_\varepsilon^{-1}\, : \,  E(\rho_\varepsilon) \to E(M_1 h_\varepsilon).
\end{equation}
The expression of $\mathcal{W}$ is of the form $\mathcal{W}(\hat{x}, v) \coloneqq (\tau(\hat{x}), W_{\hat{x}}(v))$, where 
\begin{equation}\label{defW}
    W_{\hat{x}} = \mathcal{C}_{\varepsilon, \tau(\hat{x})} \circ g_{\hat{x}} \circ \mathcal{C}_{\varepsilon, \hat{x}}^{-1}.
\end{equation}
Therefore, there exists an $\varepsilon$-slow function, still denoted by $\rho_\varepsilon$, such that $\mathcal{W} : E(\rho_\varepsilon) \to E(\rho_\varepsilon)$ and $\mathrm{Lip} \, \mathcal{W} \leqslant \e^{-\lambda_2+2\varepsilon}$ \cite[Section~2.2]{BertelootDupont2019}. Since the composition and inverse of fibered maps are still fibered, we have the following lemma.
\begin{prelem}\label{Wfibre}
    Let $f$ be a skew product such that $\lambda_2 = \lambda_p$. Then, for every $\hat x\in X$, the application $W_{\hat x}$ defined by \eqref{defW} is fibered.
\end{prelem}
\subsection{Resonances and constraints}
As previously, the Lyapunov exponents of $(f,\mu)$ are denoted by $\lambda_2 < \lambda_1$. For $j=1,2$, we define the set of \emph{$j$-resonant indices} by:
\begin{equation*}
    \mathfrak{R}_j \coloneqq \{ \alpha \in \mathbb{N}^2 : |\alpha| \geqslant2 \text{ and } \alpha_1\lambda_1+ \alpha_2\lambda_2 = \lambda_j \}.
\end{equation*}
Since $\lambda_2 < \lambda_1$, we immediately see that $\mathfrak{R}_2 = \emptyset$ and $\mathfrak{R}_1 = \{(0, k)\}$ if $\lambda_1=\lambda_2k$ with $k\geqslant 2$ or $\mathfrak{R}_1 = \emptyset$ otherwise. We fix a constant $a>0$ such that
\begin{equation}
    \alpha_1\lambda_1 + \alpha_2\lambda_2 - \lambda_j \notin [-a, a]
\end{equation}
for $j = 1,2$ and for every $\alpha \in \mathbb{N}^2 \setminus \mathfrak{R}_1$ such that $2\leqslant \lvert\alpha\lvert\,\leqslant \left[2\lambda_1/\lambda_2\right]$, where $[\cdot]$ denotes the integer part. For each $\delta > 0$, we define the \emph{shifted Lyapunov exponents} by setting: $\lambda_j^\delta \coloneqq \lambda_j - \delta.$ We then define the \emph{$(j,\delta)$-resonant sets} by:
\begin{equation*}
    \mathfrak{R}_j^\delta \coloneqq \{ \alpha \in \mathbb{N}^2 : |\alpha| \geqslant 2 \text{ and } \alpha_1\lambda_1^\delta + \alpha_2\lambda_2^\delta = \lambda_j^\delta \}.
\end{equation*}
Since $\lambda_2 < \lambda_1$, we immediately see that $\mathfrak{R}^\delta_2 = \emptyset$. One can see that $ \mathfrak{R}_1^\delta = \emptyset$ if $\delta$ is sufficiently small. There is therefore no resonance relation for these shifted Lyapunov exponents. More precisely, setting $b \coloneqq \frac{1}{2} \min\{\delta, a\}$,
we have 
\begin{equation}
    \alpha_1\lambda_1^\delta +\alpha_2\lambda_2^\delta - \lambda_j^\delta \notin [-b, b] \label{notinb}
\end{equation} for $j=1,2$ and $2 \leqslant|\alpha| \leqslant[\lambda_1^\delta/\lambda_2^\delta]$ +1.\subsection{Proof of Proposition \ref{fibre}}
\subsubsection{Construction of the coordinate changes}
Let $\rho_\varepsilon$ be the $\varepsilon$-slow function obtained at the end of Section \ref{section3.1}, provided by Definition \eqref{defcalW} of $\mathcal{W}$. Let $\hat{x} \in X$ and $n \in \mathbb{Z}$, and set
\begin{equation*}
    \rho_n \coloneqq \rho_\varepsilon(\tau^n(\hat{x})), \quad W_n \coloneqq W_{\tau^n(\hat{x})}.
\end{equation*}
The sequence $(\rho_n)_n$ is $\varepsilon$-slow and we have $W_n : \D(\rho_n) \to \D(\rho_{n+1}).$ Let $(r_n)_{n\in \Z}$ be any $\varepsilon$-slow sequence such that $r_n \leqslant \rho_n$. We set $r_{\hat{x}} \coloneqq r_0/h_\varepsilon(\hat{x})$ and $r_n^\delta \coloneqq r_n \e^{-n\delta}$. Along an orbit $\hat{x}\in X$, according to \eqref{defW}, the following diagram commutes for every $n \gg 1$:
\begin{equation}\label{2.4}
    \begin{array}{ccc}
        B(x_0, r_{\hat{x}}) & \xrightarrow{\quad f_{\hat{x}}^{-n} \quad} & B(x_{-n}, r_{\hat{x}} \e^{-n(2\varepsilon+\delta)}) \\
        \left\downarrow \scriptstyle \mathcal{C}_{\varepsilon,\hat{x}}  \right. & & \left\downarrow \scriptstyle \mathcal{C}_{\varepsilon,\tau^n(\hat{x})}\right. \\
        \D(r_0) & \xrightarrow{\quad W_{n-1} \circ \dots \circ W_0 \quad} & \D(r_n^\delta).
    \end{array}
\end{equation}
The justification of the domains of definition of the maps in this diagram is given in \cite{BertelootDupontMolino2008}.
\begin{rmk}
    According to Proposition \ref{Cepsi} and Lemma \ref{Wfibre}, we observe that all the maps appearing in diagram (\ref{2.4}) are fibered.
\end{rmk}
We want to prove Proposition \ref{fibre}. First observe that Theorem \ref{formenormale} (non-fibered normal form) is obtained by combining diagram (\ref{2.4}) with Theorem \ref{B} below, by setting $\rho_{\hat{x}} = 4r_0$, $D_{\hat{x}, n} = A_{n-1}^\delta \dots \circ A_0^\delta$  and
\begin{equation*}
    \varphi_{\hat{x}, n} = \varphi_n \circ \mathcal{C}_{\varepsilon,\tau^n(\hat{x})}.
\end{equation*}
\begin{theo}[Linearization of Chains, \cite{BertelootDupontMolino2008, BertelootDupont2019}]\label{B}
    Let $\lambda_2 < \lambda_1$ be strictly positive real numbers. Let $\delta, \varepsilon > 0$ be sufficiently small. Let $(\rho_n)_{n\in\mathbb{Z}}$ be an $\varepsilon$-slow sequence and let
    \begin{equation*}
        W_n : \D(\rho_n) \to \D(\rho_{n+1})
    \end{equation*}
    be a sequence of holomorphic maps fixing the origin in $\mathbb{C}^2$ and such that
    \begin{equation*}
        \forall n \in \mathbb{Z}, \quad \mathrm{Lip}\, W_n \leqslant\theta :\coloneqq e^{-\lambda_2+2\varepsilon}.
    \end{equation*}
    Assume that there exists a decomposition $\mathbb{C}^2 = L_1\oplus L_2$ such that, for every $n \in \mathbb{Z}$, the linear map $A_n \coloneqq d_0 W_n$ is of the form $A_{n} = \emph{diag}(A_2(n), A_1(n))$ where we have $\e^{-\lambda_i - \varepsilon} \leqslant \ |A_i(n)| \leqslant \e^{-\lambda_i + \varepsilon}$, for $i=1,2$. Then there exists an $\varepsilon$-slow sequence $(r_n)_n \leqslant(\rho_n)_n$ and a sequence of holomorphic maps
    \begin{equation*}
        \varphi_n : \D(r_n^\delta) \to \D(4r_n)\quad \text{where } r_n^\delta \coloneqq r_n e^{-n\delta}
    \end{equation*}
    such that we have:
    \begin{enumerate}
        \item the distortion estimate
            \begin{equation}\label{distorB}
                \e^{n(\delta-2\varepsilon)}|u - v| \leqslant |\varphi_n(u) - \varphi_n(v)| \leqslant \e^{n(\delta+2\varepsilon)}|u - v|,
            \end{equation}
        \item the following commuting diagram
            \begin{equation*}
                \begin{tikzcd}
                    \D(r_n^\delta) \arrow[r, "W_n"] \arrow[d, "\varphi_n"'] & \D(r_{n+1}^\delta) \arrow[d, "\varphi_{n+1}"] \\
                    \D(4r_n) \arrow[r, "A_n^\delta"'] & \D(4r_{n+1}).
                \end{tikzcd}
            \end{equation*}
    \end{enumerate}
    where the maps $A_n^\delta$ are given by $A_n^\delta \coloneqq e^\delta A_n$.
\end{theo}
In order to prove Proposition \ref{fibre}, it is enough to show that the maps $\varphi_n$ can be chosen fibered, since we know from Proposition \ref{Cepsi} that, for every $\hat{x}\in X$, the map $\mathcal{C}_{\varepsilon,\tau^n(\hat{x})}$ is fibered. This is the purpose of the following proposition.
\begin{preprop}\label{phinfibre}
    In Theorem \ref{B}, if the sequence $(W_n)_{n\in\Z}$ is assumed to be fibered, then the sequence $(\varphi_n)_{n\in\Z}$ satisfying the stated properties can be chosen to be fibered. In particular, if we write
    \begin{equation*}
        \begin{array}{ccccc}
            \varphi_{n} &: &\D(r_n^\delta) &\longrightarrow &\D(4r_n)\\
            &&(z,w)&\longmapsto &(\phi_{1,n}(z), \phi_{2,n}(z,w)),
        \end{array}
    \end{equation*}
    the maps $\phi_{1,n}$ and $\phi_{2,n}$ satisfy, for every $z,z'\in B(0,r_n^\delta)$ and every $(0,u),(0,u')\in \{0\}\times B(0,r_n^\delta)$,
    \begin{equation}\label{distbase}
        \e^{n(\delta-2\varepsilon)}|z - z'| \leqslant |\phi_{1,n}(z) - \phi_{1,n}(z')|\leqslant\e^{n(\delta+2\varepsilon)}|z - z'|,
    \end{equation}
    \begin{equation}\label{distfibre}
        \e^{n(\delta-2\varepsilon)}|u - u'|\leqslant|\phi_{2,n}(0,u) - \phi_{2,n}(0,u')|\leqslant\e^{n(\delta+2\varepsilon)}|u - u'|.
    \end{equation}
\end{preprop}
\subsubsection{Proof of Proposition \ref{phinfibre}}
In order to prove Proposition \ref{phinfibre}, we need to go back to the construction of the maps $\varphi_n$ in Theorem \ref{B}. This construction follows from one lemma and two propositions, which we state below and whose proofs can be found in \cite{BertelootDupontMolino2008, BertelootDupont2019}.
\begin{prelem}[\cite{BertelootDupont2019}, Lemma 3.1]\label{3.1}
    Let $W_n : \D(\rho_n) \to \D(\rho_{n+1})$ be a sequence of holomorphic maps satisfying the assumptions of Theorem \ref{B}. Set $\rho_n^\delta \coloneqq \rho_n \e^{-n\delta}$, $\Delta_n \coloneqq \e^{n\delta}\mathrm{Id}_{\mathbb{C}^2}$, and
    \begin{equation}
         W_n^\delta \coloneqq \Delta_{n+1} \circ W_n \circ \Delta_n^{-1}. \label{wdelta}
    \end{equation}
    Then $\mathrm{Lip}\, W_n^\delta = \theta \e^\delta < 1$ and the linear part $A_n^\delta \coloneqq \e^\delta A_n$ of $W_n^\delta$ satisfies
    \begin{equation*}
        \e^{-\lambda_1 - \varepsilon + \delta} |v| \leqslant|A_n^\delta(v)| \leqslant\e^{-\lambda_2 + \varepsilon + \delta} |v| \quad \text{for all } v \in \mathbb{C}^2.
    \end{equation*}
    Moreover, the following diagram commutes:
    \begin{equation*}
        \begin{array}{c}
            \D(\rho_n^\delta) \xrightarrow{\quad W_n \quad} \D(\rho_{n+1}^\delta)  \\
            \Big\downarrow \Delta_n \qquad\qquad \Big\downarrow \Delta_{n+1} \\
            \D(\rho_n) \xrightarrow{\quad W_n^\delta \quad} \D(\rho_{n+1}).
        \end{array}
    \end{equation*}
\end{prelem}
\begin{preprop}[\cite{BertelootDupont2019}, Proposition 3.2]\label{3.2}
    Let $W_n^\delta : \D(\rho_n) \to \D(\rho_{n+1})$ be the sequence of holomorphic maps given by Lemma \ref{3.1}. There exists an $\varepsilon$-slow sequence $(r_n)_{n\in\Z} \leqslant (\rho_n)_{n\in\Z}$ and a sequence of holomorphic maps
    \begin{equation*}
        T^1_n : \D(r_n)\longrightarrow\D(2r_n)
    \end{equation*}
    such that:
    \begin{enumerate}
        \item $d_0 T_n^1 = \mathrm{Id}_{\mathbb{C}^2}$;
        \item for every $n\in\Z$ and every $(u,u') \in \D(r_n) \times \D(r_n)$, we have
            \begin{equation}\label{distT1}
                \e^{-\varepsilon}\di(u, u') \leqslant|T_{n}^1(u) - T_{n}^1(u')| \leqslant\e^{\varepsilon}\di(u, u');
            \end{equation}
        \item the following diagram commutes
            \begin{equation*}
                \begin{array}{c}
                    \D(r_n) \xrightarrow{\quad W_n^\delta \quad} \D(r_{n+1}) \\
                    \Big\downarrow T_n^1 \qquad\qquad\quad \Big\downarrow T_{n+1}^1 \\
                    \D(2r_n) \xrightarrow{\quad X_n \quad} \D(2r_{n+1}),
                \end{array}
            \end{equation*}
            where $X_n = A_n^\delta + O([\lambda_1^\delta/\lambda_2^\delta] + 2)$ for every $n \in \mathbb{Z}$.
    \end{enumerate}
\end{preprop}
\begin{preprop}[\cite{BertelootDupont2019}, Proposition 3.4]\label{3.4}
    Let $(X_n)_{n\in\mathbb{Z}}$ be an $\varepsilon$-slow sequence of holomorphic maps with linear parts $(A_n^\delta)_{n\in\mathbb{Z}}$. Assume that there exist $0 < m \leqslant M < 1$ such that, for every $n\in\Z$ and every $v\in\C^2$, $m|v| \leqslant|A_n^\delta(v)| \leqslant M|v|$ and that $X_n = A_n^\delta + O(q + 1)$ where $q \geqslant1$ satisfies $(M \e^{2\varepsilon})^{q+1} < m \e^{-\varepsilon}. $ Then there exists an $\varepsilon$-slow sequence $(r_n)_{n\in\mathbb{Z}}$ and a sequence $(T_n^2)_{n\in\mathbb{Z}}$ of holomorphic maps
    \begin{equation*}
        T^2_n : \D(r_n) \longrightarrow \D(2r_n)
    \end{equation*}
    such that:
    \begin{enumerate}
        \item $d_0 T_n^2 = \mathrm{Id}_{\C^2}$;
        \item for every $n \in \Z$ and every $(u,u') \in \D(r_n) \times \D(r_n)$,
            \begin{equation}\label{distT2}
                \e^{-\varepsilon}\di(u, u') \leqslant|T_n^2(u) - T_n^2(u')| \leqslant\e^{\varepsilon}\di(u, u'),
            \end{equation}
        \item the following diagram commutes:
            \begin{equation*}
                \begin{array}{c}
                    \D(r_n) \xrightarrow{\quad X_n \quad} \D(r_{n+1})  \\
                    \Big\downarrow T_n^2 \qquad\qquad\quad \Big\downarrow T_{n+1}^2 \\
                    \D(2r_n) \xrightarrow{\quad A_n^\delta \quad} \D(2r_{n+1}).
                \end{array}
            \end{equation*}
    \end{enumerate}
\end{preprop}
By applying successively Lemma \ref{3.1} and Propositions \ref{3.2} and \ref{3.4}, the sequence of maps $(\varphi_n)_{n\geqslant 0}$ given by Theorem \ref{B} is defined by
\begin{equation*}
    \varphi_n \coloneqq T_n^2 \circ T_n^1 \circ \Delta_n.
\end{equation*}
We can now prove Proposition \ref{phinfibre}. We say that a sequence $(G_n)_{n\in\mathbb{Z}}$ of holomorphic maps $G_n : \D(a_n) \to \D(a_{n+1})$ is \emph{$\varepsilon$-slow} if $(a_n)_{n\in\mathbb{Z}}$ is an $\varepsilon$-slow sequence (see \eqref{lente}).
\begin{proof}[Proof of Proposition \ref{phinfibre}]
    We study the construction of $\Delta_n$, $T^1_n$ and $T^2_n$.
    \paragraph{Fact 1.} In Lemma \ref{3.1}, if $(W_n)_{n\in\Z}$ is assumed to be fibered, then $(W_n^\delta)_{n\in\Z}$ is also fibered.
    \begin{proof}
        The maps $\Delta_n$ are homotheties. The result follows from the definition (\ref{wdelta}) of $W_n^\delta$.
    \end{proof}
    \paragraph{Fact 2.} In Proposition \ref{3.2}, if $(W_n^\delta)_{n\in\Z}$ is assumed to be fibered, then the sequence $(T^1_n)_{n\in\Z}$ satisfying the stated properties can be chosen to be fibered.
    \begin{proof}
        We follow the proof of Proposition 3.2 in \cite{BertelootDupont2019}. Let $(G_n)_{n\in\mathbb{Z}}$ be a chain such that $d_0G_n = A_n^\delta$. Denote by $G_n^{(k)}$ the homogeneous part of degree $k$ of $G_n$. Suppose that there exists a sequence of homogeneous polynomial maps $\left(H_n^{(k)}\right)_{n\in\Z}$ of degree $k$ satisfying the homological equation
        \begin{equation}\label{eqhomol}
            G_n^{(k)} + H_{n+1}^{(k)}\circ A_n^\delta - A_n^\delta\circ H_n^{(k)} = 0.
        \end{equation}
        Then, if we set
        \begin{equation}\label{defsnk}
            S_n^{(k)} \coloneqq \mathrm{Id} + H_n^{(k)},
        \end{equation}
        the sequence $S_{n+1}^{(k)} \circ G_n \circ \left(S_n^{(k)}\right)^{-1}$ has no homogeneous part of order $k$, see \cite{BertelootDupont2019}. Proposition \ref{3.2} is then obtained by successively applying this procedure a finite number of times, starting with $G_n = W_n^{\delta}$ and $k=2$, then with $G_n = S_{n+1}^{(2)} \circ W_n^{\delta} \circ \left(S_n^{(2)}\right)^{-1}$ and $k=3$, and so on, up to $k = [\lambda_1^\delta/\lambda_2^\delta]+1$. 
        
        We will show that if $\left(G_n\right)_{n\in\Z}$ is fibered, then for every $k\geqslant 2$ there exists a solution $\left(H_n^{(k)}\right)_{n\in\Z}$ to \eqref{eqhomol} which is fibered. The change of coordinates $S_n^{(k)}$ defined in \eqref{defsnk} is then fibered, as is the composition $(T_n^1)_{n\in\Z}$ of these changes of coordinates, which will complete the proof of Fact 2.

        By linearity, it is enough to solve \eqref{eqhomol} when the sequence $\left(G_n^{(k)}\right)_{n\in\Z}$ has only one nonzero coordinate, which is moreover a monomial. Since $(G_n)_{n\in\Z}$ is fibered, and according to the resonance conditions \eqref{notinb}, there are three cases to consider for each $k\geqslant 2$:
        \begin{enumerate}
            \item $G_n^{(k)}(z,w) = (a_n z^k,\, 0)$;
            \item $G_n^{(k)}(z,w) = (0,\, a_n z^{\alpha_2} w^{\alpha_1})$ where $\alpha_1 \lambda_1^{\delta} + \alpha_2 \lambda_2^{\delta} - \lambda_1^{\delta} > b$;
            \item $G_n^{(k)}(z,w) = (0,\, a_n z^{\alpha_2} w^{\alpha_1})$ where $\alpha_1 \lambda_1^{\delta} + \alpha_2 \lambda_2^{\delta} - \lambda_1^{\delta} < -b$.
        \end{enumerate}
        Let $\mathcal H_k$ be the vector space of homogeneous fibered polynomial maps of degree $k$. Since $A_n^\delta$ and $(A_n^\delta)^{-1}$ are diagonal, composition with these maps preserves $\mathcal H_k$. To solve \eqref{eqhomol}, it suffices to show that the following linear operator $\Gamma$, defined on the space of $\varepsilon$-slow sequences in $\mathcal{H}_k$, is surjective:
        \begin{equation*}
            \Gamma : \left(H_n^{(k)}\right)_{n \in \mathbb{Z}} \in \left(\mathcal{H}_k\right)^{\Z} \mapsto \left(H_{n+1}^{(k)} \circ A_n^\delta - A_n^\delta \circ H_n^{(k)}\right)_{n \in \mathbb{Z}}.
        \end{equation*}\\
        \noindent\underline{Case 1}: We set
            \begin{equation*}
                H_n^{(k)} \coloneqq (A_n^\delta)^{-1} G_n^{(k)}+ \sum_{r \geqslant1}(A_n^\delta)^{-1} \dots (A_{n+r}^\delta)^{-1}G_{n+r}^{(k)}A_{n+r-1}^\delta \dots A_n^\delta.
            \end{equation*}
            This series converges because its second coordinate is zero and its first coordinate is equivalent to
            \begin{equation*}
                \sum_{r\geqslant 1} \e^{r\lambda_2^{\delta}}\, a_{n+r} \left(\e^{-r\lambda_2^{\delta}}\right)^{k} \simeq \sum_{r\geqslant 1} a_{n+r}\, \e^{-r\lambda_2^{\delta}(k-1)}
            \end{equation*}
            since the sequence $(a_n)_{n\in\Z}$ is $\varepsilon$-slow.\\[5pt]
        \noindent \underline{Case 2}: We use the same formula as in Case 1. The series converges because its first coordinate is zero and its second coordinate is equivalent to
            \begin{equation*}
                \sum_{r\geqslant 1} \e^{r\lambda_1^{\delta}} a_{n+r}\left(\e^{-r\lambda_2^{\delta}}\right)^{\alpha_2} \left(\e^{-r\lambda_1^{\delta}}\right)^{\alpha_1}\simeq\sum_{r\geqslant 1} a_{n+r}\underset{<\e^{-rb}}{\underbrace{\e^{-r\left(\alpha_1 \lambda_1^{\delta} + \alpha_2 \lambda_2^{\delta} - \lambda_1^{\delta}\right)}}}.
            \end{equation*}\\
        \noindent\underline{Case 3}: Here we set
            \begin{equation*}
                H_n^{(k)} \coloneqq -G_{n-1}^{(k)} (A_{n-1}^\delta)^{-1} - \sum_{r \geqslant1} A_{n-1}^\delta \dots A_{n-r}^\delta G_{n-(r+1)}^{(k)}(A_{n-(r+1)}^\delta)^{-1} \dots(A_{n-1}^\delta)^{-1}.
            \end{equation*}
            The series converges because its first coordinate is zero and its second coordinate is equivalent to
            \begin{equation*}
                \sum_{r\geqslant 1} \e^{-r\lambda_1^{\delta}} a_{n-(r+1)} \left(\e^{r\lambda_2^{\delta}}\right)^{\alpha_2} \left(\e^{r\lambda_1^{\delta}}\right)^{\alpha_1}\simeq\sum_{r\geqslant 1} a_{n-(r+1)}\underset{<\e^{-rb}}{\underbrace{\e^{r\left(\alpha_1 \lambda_1^{\delta} + \alpha_2 \lambda_2^{\delta} - \lambda_1^{\delta}\right)}}}.
            \end{equation*}
        In all three cases, one checks that
        \begin{equation*}
            \left(G_n^{(k)}\right)_{n\in\Z}+\Gamma\left(\left(H_n^{(k)}\right)_{n\in\Z}\right)= 0.
        \end{equation*}
        Moreover, $H_n^{(k)}$ is indeed fibered: in Case 1, its first coordinate is a multiple of $z^k$, while in Cases 2 and 3, its first coordinate is zero. This completes the proof of Fact 2.
    \end{proof}
    \paragraph{Fact 3.} In Proposition \ref{3.4}, if $(X_n)_{n\in\Z}$ is assumed to be fibered, then the sequence $(T^2_n)_{n\in\Z}$ satisfying the stated properties can be chosen to be fibered.
    \begin{proof}
        We again follow \cite[Proposition~3.4]{BertelootDupont2019}. For $n\in\Z$, the map $T^2_n$ is defined there by
        \begin{equation*}
            T_n^2  \coloneqq \lim_{k\to+\infty}(A_{n}^\delta)^{-1}\circ\cdots\circ (A_{n+k}^\delta)^{-1}\circ X_{n+k}\circ\cdots\circ X_n.
        \end{equation*}
        By its definition in Theorem \ref{B}, $A_n^\delta$ is fibered (since it is diagonal), and in our case $X_n$ is assumed to be fibered. Therefore, $T^2_n$ remains fibered, since a locally uniform limit of fibered maps is fibered.
    \end{proof}
    By Fact 1, $W_n^\delta$ is fibered, and hence Fact 2 implies that $T^1_n$ is also fibered. Moreover, by the commutativity of the diagram in Proposition \ref{3.2}, the maps $X_n$ are fibered, and thus, by Fact 3, $T^2_n$ is fibered as well. Therefore,
    \begin{equation*}
        \varphi_n = T^2_n\circ T^1_n\circ\Delta_n
    \end{equation*}
    remains fibered. Let us observe that restricting all the maps we have constructed to their first coordinate forms a one-dimensional chain that verify the one-dimensional version of \cite{BertelootDupont2019} In particular, since $d_0T^1_n = d_0T^2_n =\mathrm{Id}_{\C^2}$ and $\Delta_n$ is a homothety, we have $d_0\varphi_n = \e^{n\delta}\mathrm{Id}_{\C^2}.$ Hence, the writing $\varphi_n(z,w) = (\phi_{1,n}(z), \phi_{2,n}(z,w))$ gives $d_0\phi_{1,n} = \e^{n\delta}\mathrm{Id}_\C$ and using the estimates \eqref{distT1} and \eqref{distT2}, we obtain \eqref{distbase}. Equation \eqref{distfibre} is simply a particular case of the global distortion estimate \eqref{distorB} given by Theorem \ref{B}.
\end{proof}

\printbibliography

@article{BarreiraPesinSchmeling1999,
author    = {Barreira, Luis and Pesin, Yakov and Schmeling, J{\"o}rg},
title     = {Dimension and product structure of hyperbolic measures},
year      = {1999},
journal   = {Annals of Mathematics},
volume    = {149},
series    = {2},
number    = {3},
pages     = {755-783},
doi       = {10.2307/121072}
}

@article{BertelootDupont2019,
author  = {Berteloot, Fran{\c{c}}ois and Dupont, Christophe},
title   = {A distortion theorem for the iterated inverse branches of a holomorphic endomorphism of $\mathbb{P}^k$},
year    = {2019},
journal = {London Mathematical Society},
volume  = {99},
number  = {1},
pages   = {153-172},
doi     = {10.1112/jlms.12163}
}

@article{BertelootDupontMolino2008,
author = {Berteloot, Fran{\c{c}}ois and Dupont, Christophe and Molino, Laura},
title = {Normalization of bundle holomorphic contractions and applications to dynamics},
year = {2008},
journal = {Annales de l'Institut Fourier},
volume = {58},
number = {6},
pages = {2137-2168},
doi = {10.5802/aif.2409}
}

@article{BinderDeMarco2003,
author    = {Binder, Ilia and DeMarco, Laura},
title     = {Dimension of pluriharmonic measure and polynomial endomorphisms of $\mathbb{C}^n$},
year      = {2003},
journal   = {International Mathematics Research Notices},
volume    = {2003},
number    = {11},
pages     = {613-625},
doi       = {10.1155/S107379280320910X}
}

@article{CornfeldFominSinai1982,
author    = {Cornfeld, Isaac P. and Fomin, Sergej V. and Sinai, Yakov G.},
title     = {Ergodic Theory},
year      = {1982},
journal    = {Grundlehren der mathematischen Wissenschaften},
volume    = {245},
number   = {1},
pages     = {227-291},
doi = {10.1007/978-1-4615-6927-5}
}

@article{DinhDupont2004,
author    = {Dinh, Tien-Cuong and Dupont, Christophe},
title     = {Dimension de la mesure d'équilibre d'applications méromorphes},
year      = {2004},
journal   = {Journal of Geometric Analysis},
volume    = {14},
number    = {4},
pages     = {613-627},
doi       = {10.1007/BF02922143}
}

@article{DinhSibony2010,
author    = {Dinh, Tien-Cuong and Sibony, Nessim},
title     = {Dynamics in several complex variables: endomorphisms of projective spaces and polynomial-like mappings},
year      = {2010},
journal    = {Lecture Notes in Mathematics},
volume    = {1998},
pages     = {165-294},
doi       = {10.1007/978-3-642-13171-4_4}
}

@article{Dupont2011,
author={Dupont, Christophe},
title={On the dimension of invariant measures of endomorphisms of $\mathbb{CP}^k$},
year={2011},
journal={Mathematische Annalen},
volume={349},
number={3},
pages={509-528},
doi = {10.1007/s00208-010-0519-1}
}

@article{DupontRogue2020,
author = {Dupont, Christophe and Rogue, Axel},
title = {Dimension of ergodic measures and currents on $\mathbb{CP}^2$},
year = {2020},
journal = {Ergodic Theory and Dynamical Systems},
volume = {40},
number = {8},
pages = {2131-2155},
doi = {10.1017/etds.2018.137}
}

@article{DupontTaflin2021,
author    = {Dupont, Christophe and Taflin, Johan},
title     = {Dynamics of fibered endomorphisms of $\mathbb{P}^k$},
year      = {2021},
journal   = {Annali della Scuola Normale Superiore di Pisa},
series    = {5},
volume    = {22},
number    = {4},
pages     = {1717-1754},
doi       = {10.2422/2036-2145.201911_008}
}

@article{FornaessSibony2001,
author = {Fornaess, John E. and Sibony, Nessim},
title = {Some open problems in higher dimensional complex analysis and complex dynamics},
year = {2001},
journal = {Publicacions Matemàtiques},
volume = {45},
number = {2},
pages = {529-547},
doi = {10.5565/PUBLMAT_45201_11}
}

@article{Jonsson1999,
author={Jonsson, Mattias},
title={Dynamics of polynomial skew products on $\mathbb{C}^2$},
year={1999},
journal={Mathematische Annalen},
volume={314},
number={3},
pages={403-447},
doi = {10.1007/s002080050301}
}

@article{JonssonVarolin2002,
author    = {Jonsson, Mattias and Varolin, Dror},
title     = {Stable manifolds of holomorphic diffeomorphisms},
year      = {2002},
journal   = {Inventiones mathematicae},
volume    = {149},
number    = {2},
pages     = {409-430},
doi       = {10.1007/s002220200220}
}

@article{KatokHasselblatt1995,
author={Katok, Anatole and Hasselblatt, Boris},
title={Introduction to the Modern Theory of Dynamical Systems},
year={1995},
journal={Cambridge University Press},
volume={54},
pages={660-666},
doi={10.1017/CBO9780511809187},
}

@article{Ledrappier1981,
author   = {Ledrappier, Fran{\c{c}}ois},
title    = {Some relations between dimension and {L}yapounov exponents},
year     = {1981},
journal  = {Communications in Mathematical Physics},
volume   = {81},
number   = {2},
pages    = {229-238},
doi = {10.1007/BF01208896}
}

@article{LedrappierYoung1985,
author = {Ledrappier, Fran{\c{c}}ois and Young, Lai-Sang},
title = {The metric entropy of diffeomorphisms. {II}: Relations between entropy, exponents and dimension.},
year = {1985},
journal = {Annals of Mathematics. Second Series},
volume = {122},
pages = {509-539},
doi = {10.2307/1971328}
}

@article{Mane1988,
author     = {Ma{\~n}{\'e}, Ricardo},
title      = {The {H}ausdorff dimension of invariant probabilities of rational maps},
year      = {1988},
journal = {Lecture Notes in Mathematics},
volume    = {1331},
pages     = {86-117},
doi       = {10.1007/BFb0083061}
}

@article{Manning1984,
author = {Anthony Manning},
title = {The Dimension of the Maximal Measure for a Polynomial Map},
year = {1984},
journal = {Annals of Mathematics},
volume = {119},
number = {2},
pages = {425-430},
doi = {10.2307/2007044}
}

@article{Pesin1998,
author  = {Pesin, Yakov B.},
title   = {Dimension Theory in Dynamical Systems: Contemporary Views and Applications},
journal = {Ergodic Theory and Dynamical Systems},
year    = {1998},
volume  = {18},
number  = {4},
pages   = {1043--1045},
doi     = {10.1017/S0143385798128298}
}

@article{Rokhlin1949,
title={On the fundamental ideas of measure theory},
author={Rokhlin, Vladimir A.},
year={1949},
journal={Matematicheskii Sbornik},
volume={67},
number={1},
pages={107-150}
}

@article{Rokhlin1967,
title={Lectures on the entropy theory of measure-preserving transformations},
author={Rokhlin, Vladimir A.},
year={1967},
journal={Russian Mathematical Surveys},
volume={22},
number={5},
pages={1-52},
doi = {10.1070/RM1967v022n05ABEH001224}
}

@article{DeThelinVigny2015,
author = {de Th{\'e}lin, Henry and Vigny, Gabriel},
title = {On the measures of large entropy on a positive closed current},
year = {2015},
journal = {Mathematische Zeitschrift},
volume = {280},
number = {3-4},
pages = {919-944},
doi = {10.1007/s00209-015-1456-1}
}

@article{Young1982,
author={Young, Lai-Sang},
title={Dimension, entropy and {L}yapunov exponents},
year={1982},
journal={Ergodic Theory and Dynamical Systems},
volume={2},
number={1},
pages={109–124},
doi={10.1017/S0143385700009615}
}

\noindent\footnotesize{Clara Genes\\Universit\'e de Rennes\\IRMAR - UMR CNRS 6625\\35042 Campus de Beaulieu, Rennes, France\\clara.genes@univ-rennes.fr}
\end{document}